\documentclass[a4paper,twoside]{article}

\def\shorttitle{Strong Maximum Principle for Fractional Diffusion Equations}
\def\shortauthor{Y. Liu, W. Rundell and M. Yamamoto}

\usepackage{amsfonts}
\usepackage{amsmath}
\usepackage{amssymb}
\usepackage{amscd}
\usepackage{amsbsy}
\usepackage{amsthm}
\usepackage{bm}
\usepackage{empheq}
\usepackage{graphicx}
\usepackage{psfrag}
\usepackage{color}
\usepackage{cite}
\usepackage{multicol}

\allowdisplaybreaks[4]

\newfont{\myfnt}{cmssi10 scaled 1440}
\numberwithin{equation}{section}
\catcode`@=11
\def\ps@nk{\def\@oddhead{\vbox{\hbox to \hsize{\pic \footnotesize \it \shorttitle
\hfill \rm \thepage} \vspace{1mm} \vspace*{-2mm}}}
\def\@evenhead{\vbox{\hbox to \hsize{\pic \footnotesize \rm \thepage \hfill \it \shortauthor}
\vspace{1mm} \vspace*{-2mm}}}
\def\@oddfoot{} \def\@evenfoot{}}
\def\ps@first{\def\@oddhead{\vbox{\hbox to \hsize{\pic \footnotesize
} \break}}
\def\@oddfoot{} \def\@evenfoot{}}
\newtheoremstyle{thmstyle}
  {6pt}
  {6pt}
  {\it}
  {}
  {\bf}
  {}
  {.5em}
  {}
\newtheoremstyle{remstyle}
  {6pt}
  {6pt}
  {\rm}
  {}
  {\bf}
  {}
  {.5em}
  {}
\def\Section#1{\Sec{\large #1} \setcounter{equation}{0} \vskip -6mm \indent}
\def\Sec{\@Startsection{section}{1}{\z@}
                                   {-3.5ex \@plus -1ex \@minus -.2ex}%
                                   {2.3ex \@plus.2ex}%
                                   {\normalfont\large\bfseries\boldmath}}
\def\@Startsection#1#2#3#4#5#6{%
  \if@noskipsec \leavevmode \fi
  \par
  \@tempskipa #4\relax
  \@afterindenttrue
  \ifdim \@tempskipa <\z@
    \@tempskipa -\@tempskipa \@afterindentfalse
  \fi
  \if@nobreak
    \everypar{}%
  \else
    \addpenalty\@secpenalty\addvspace\@tempskipa
  \fi
  \@ifstar
    {\@ssect{#3}{#4}{#5}{#6}}%
    {\@dblarg{\@Sect{#1}{#2}{#3}{#4}{#5}{#6}}}}
\def\@Sect#1#2#3#4#5#6[#7]#8{%
  \ifnum #2>\c@secnumdepth
    \let\@svsec\@empty
  \else
    \refstepcounter{#1}%
    \protected@edef\@svsec{\@seccntformat{#1}\relax}%
  \fi
  \@tempskipa #5\relax
  \ifdim \@tempskipa>\z@
    \begingroup
      #6{%
          \@hangfrom{\hskip #3\relax\@svsec \hskip -2.5mm}%
          \interlinepenalty \@M #8\@@par}
    \endgroup
    \csname #1mark\endcsname{#7}%
    \addcontentsline{toc}{#1}{%
      \ifnum #2>\c@secnumdepth \else
        \protect\numberline{\csname the#1\endcsname}%
      \fi
      #7}%
  \else
    \def\@svsechd{%
      #6{\hskip #3\relax
      \@svsec #8}%
      \csname #1mark\endcsname{#7}%
      \addcontentsline{toc}{#1}{%
        \ifnum #2>\c@secnumdepth \else
          \protect\numberline{\csname the#1\endcsname}%
        \fi
        #7}}%
  \fi
  \@xsect{#5}}
\renewenvironment{abstract}{%
        \small
        \quotation
         \noindent {\bfseries \abstractname } }%
      {\if@twocolumn\else\endquotation\fi}

\def\Subsec{\@StartSubsection{subsection}{2}{\z@}%
                                     {-3.25ex\@plus -1ex \@minus -.2ex}%
                                     {1.5ex \@plus .2ex}%
                                     {\normalfont\normalsize\bfseries\boldmath}}
\def\@StartSubsection#1#2#3#4#5#6{%
  \if@noskipsec \leavevmode \fi
  \par
  \@tempskipa #4\relax
  \@afterindenttrue
  \ifdim \@tempskipa <\z@
    \@tempskipa -\@tempskipa \@afterindentfalse
  \fi
  \if@nobreak
    \everypar{}%
  \else
    \addpenalty\@secpenalty\addvspace\@tempskipa
  \fi
  \@ifstar
    {\@ssect{#3}{#4}{#5}{#6}}%
    {\@dblarg{\@SubSect{#1}{#2}{#3}{#4}{#5}{#6}}}}
\def\@SubSect#1#2#3#4#5#6[#7]#8{%
  \ifnum #2>\c@secnumdepth
    \let\@svsec\@empty
  \else
    \refstepcounter{#1}%
    \protected@edef\@svsec{\@seccntformat{#1}\relax}%
  \fi
  \@tempskipa #5\relax
  \ifdim \@tempskipa>\z@
    \begingroup
      #6{%
          \@hangfrom{\hskip #3\relax\@svsec\hskip -1.5mm}%
          \interlinepenalty \@M #8\@@par}
    \endgroup
    \csname #1mark\endcsname{#7}%
    \addcontentsline{toc}{#1}{%
      \ifnum #2>\c@secnumdepth \else
        \protect\numberline{\csname the#1\endcsname}%
      \fi
      #7}%
  \else
    \def\@svsechd{%
      #6{\hskip #3\relax
      \@svsec #8}%
      \csname #1mark\endcsname{#7}%
      \addcontentsline{toc}{#1}{%
        \ifnum #2>\c@secnumdepth \else
          \protect\numberline{\csname the#1\endcsname}%
        \fi
        #7}}%
  \fi
  \@xsect{#5}}
\def\list#1#2{\ifnum \@listdepth >5\relax \@toodeep \else \global
\advance \@listdepth\@ne \fi \rightmargin \z@ \listparindent\z@
\itemindent\z@ \csname @list\romannumeral\the\@listdepth\endcsname
\def\@itemlabel{#1}\let\makelabel\@mklab \@nmbrlistfalse #2\relax
\@trivlist \parskip 0pt \parindent\listparindent \advance \linewidth
-\rightmargin \advance\linewidth -\leftmargin \advance\@totalleftmargin
\leftmargin \parshape \@ne \@totalleftmargin \linewidth \ignorespaces}
\renewcommand{\@makecaption}[2]{\begin{center}#1. #2\end{center}}
\catcode`@=12 
\theoremstyle{thmstyle}
\newtheorem{thm}{\indent Theorem}[section]
\newtheorem{lem}[thm]{\indent Lemma}

\newtheorem{coro}[thm]{\indent Corollary}

\newtheorem{prob}[thm]{\indent Problem}
\theoremstyle{remstyle}
\newtheorem{rem}[thm]{\indent Remark}

\newsavebox{\mygraphic}
\def\pic{\begin{picture}(0,0) \put(-210,-1250){\usebox{\mygraphic}} \end{picture}}
\newfont{\HUGEbf}{cmbx10 scaled 3500}
\definecolor{gray}{rgb}{0.9,0.9,0.9}
\def\thebibliography#1{\section*{\bf \large References}
\list{[\arabic{enumi}]} {\settowidth \labelwidth{[#1]} \leftmargin
\labelwidth \advance \leftmargin \labelsep \usecounter{enumi}}
\def\newblock{\hskip .11em plus .33em minus .07em} \footnotesize \sloppy \clubpenalty
4000 \widowpenalty 4000 \sfcode`\.=1000 \relax}

\def\BC{\mathbb C}
\def\BN{\mathbb N}
\def\BR{\mathbb R}
\def\cA{\mathcal A}
\def\cD{\mathcal D}
\def\cE{\mathcal E}

\def\rd{\mathrm d}
\def\rRe{\mathrm{Re}}

\def\e{\mathrm e}

\def\supp{\mathrm{supp}}

\def\Ga{\Gamma}

\def\Om{\Omega}
\def\al{\alpha}
\def\be{\beta}
\def\ga{\gamma}
\def\de{\delta}

\def\ve{\varepsilon}

\def\ze{\zeta}

\def\la{\lambda}

\def\vp{\varphi}
\def\om{\omega}
\def\f{\frac}

\def\ov{\overline}
\def\pa{\partial}

\def\wt{\widetilde}

\theoremstyle{definition}

\numberwithin{equation}{section}

\title{\Large\bf\boldmath Strong Maximum Principle for Fractional Diffusion\\
Equations and an Application to\\
an Inverse Source Problem$^*$}

\author{\large Yikan LIU$^\dag$\qquad William RUNDELL$^\ddag$\qquad Masahiro
YAMAMOTO$^\dag$}

\date{}

\begin{document}

\maketitle

\thispagestyle{first}
\renewcommand{\thefootnote}{\fnsymbol{footnote}}

\footnotetext{\hspace*{-5mm} \begin{tabular}{@{}r@{}p{14cm}@{}} &
Manuscript last updated: \today.\\
$^\dag$ & Graduate School of Mathematical Sciences, The University of Tokyo,
3-8-1 Komaba, Meguro-ku, Tokyo 153-8914, Japan. E-mail: ykliu@ms.u-tokyo.ac.jp,
myama@ms.u-tokyo.ac.jp\\
$^\ddag$ & Department of Mathematics, Texas A\&M University, College Station,
TX 77843-3368, USA.\\
& E-mail: rundell@math.tamu.edu\\
$^*$ & Y. Liu and M. Yamamoto have been partly supported by the A3 Foresight
Program ``Modeling and Computation of Applied Inverse Problems'' by Japan
Society of the Promotion of Science.
\end{tabular}}

\renewcommand{\thefootnote}{\arabic{footnote}}

\begin{abstract}
The strong maximum principle is a remarkable characterization of parabolic
equations, which is expected to be partly inherited by fractional diffusion
equations. Based on the corresponding weak maximum principle, in this paper we
establish a strong maximum principle for time-fractional diffusion equations
with Caputo derivatives, which is slightly weaker than that for the parabolic
case. As a direct application, we give a uniqueness result for a related
inverse source problem on the determination of the temporal component of the
inhomogeneous term.

\vskip 4.5mm

\noindent\begin{tabular}{@{}l@{ }p{10cm}} {\bf Keywords } & Fractional
diffusion equation, Caputo derivative,\\
& Strong maximum principle, Mittag-Leffler function,\\
& Inverse source problem, Fractional Duhamel's principle
\end{tabular}

\vskip 4.5mm

\noindent{\bf AMS Subject Classifications } 35R11, 26A33, 35B50, 35R30

\end{abstract}

\baselineskip 14pt

\setlength{\parindent}{1.5em}

\setcounter{section}{0}

\Section{Introduction and main results}\label{sec-intro}

Let $\Om\subset\BR^d$ ($d=1,2,3$) be an open bounded domain with a smooth
boundary (for example, of $C^\infty$ class), $T>0$ and $0<\al<1$. Consider the
following initial-boundary value problem for a time-fractional diffusion
equation
\begin{equation}\label{eq-ibvp-u}
\begin{cases}
\pa_t^\al u(x,t)+\cA u(x,t)=F(x,t) & (x\in\Om,\ 0<t\le T),\\
u(x,0)=a(x) & (x\in\Om),\\
u(x,t)=0 & (x\in\pa\Om,\ 0<t\le T),
\end{cases}
\end{equation}
where $\pa_t^\al$ denotes the Caputo derivative defined by
\[
\pa_t^\al f(t):=\f1{\Ga(1-\al)}\int_0^t\f{f'(s)}{(t-s)^\al}\,\rd s,
\]
and $\Ga(\,\cdot\,)$ denotes the Gamma function. Here $\cA$ is an elliptic
operator defined for $f\in\cD(\cA):=H^2(\Om)\cap H^1_0(\Om)$ as
\begin{equation}\label{eq-def-A}
\cA f(x)=-\sum_{i,j=1}^d\pa_j(a_{ij}(x)\pa_if(x))+c(x)f(x)\quad(x\in\Om),
\end{equation}
where $a_{ij}=a_{ji}$ ($1\le i,j\le d$) and $c\ge0$ in $\ov\Om\,$. Moreover, it
is assumed that $a_{ij}\in C^1(\ov\Om)$, $c\in C(\ov\Om)$ and there exists a
constant $\de>0$ such that
\[
\de\sum_{i=1}^d\xi_i^2\le\sum_{i,j=1}^da_{ij}(x)\xi_i\xi_j\quad
(\forall\,x\in\ov\Om\,,\ \forall\,(\xi_1,\ldots,\xi_d)\in\BR^d).
\]
The assumptions on the initial data $a$ and the source term $F$ will be
specified later.

Fractional diffusion equations, especially the governing equation in
\eqref{eq-ibvp-u} with a Caputo derivative in time, have been widely used as
model equations for describing the anomalous diffusion phenomena in highly
heterogeneous aquifer and complex viscoelastic material (see
\cite{AG92,GCR92,HH98,MK00,N86}). Due to its practical applications,
\eqref{eq-ibvp-u} has drawn extensive attentions of mathematical researchers
during the recent years. In Luchko \cite{L10}, the generalized solution to
\eqref{eq-ibvp-u} with $F=0$ was represented by means of the Mittag-Leffler
function, and the unique existence of the solution was proved. Sakamoto and
Yamamoto \cite{SY11} investigated the well-posedness and the asymptotic
behavior of the solution to \eqref{eq-ibvp-u}. Very recently, Gorenflo et al.
\cite{GLY15} re-defined the Caputo derivative in the fractional Sobolev spaces
and investigated \eqref{eq-ibvp-u} from the viewpoint of the operator theory.
Regarding numerical treatments, we refer e.g.\! to \cite{LZATB07,MT04} for the
finite difference method and \cite{JLPZ13,JLZ13} for the finite element method.
Meanwhile, \eqref{eq-ibvp-u} has also gained population among the inverse
problem school; recent literatures include \cite{LRYZ13,RXZ13,JR15}. Here we do
not intend to enumerate a complete list of related works. It reveals in the
existing works that fractional diffusion equations show certain similarities to
classical parabolic equations (i.e., $\al=1$ in \eqref{eq-ibvp-u}), whereas
also diverge considerably from their integer prototypes in the senses of the
limited smoothing property in space and slow decay in time.

Other than the above mentioned aspects, the maximum principle is also one of
the remarkable characterizations of parabolic equations, which is not only
significant by itself but also applicable in many related problems. However,
researches especially on the strong maximum principle for time-fractional
diffusion equations with Caputo derivatives are inadequate due to the technical
difficulties in treating the fractional derivatives. Luchko \cite{L09}
established a weak maximum principle for \eqref{eq-ibvp-u} by a key estimate of
the Caputo derivative at an extreme point, by which the uniqueness of a
classical solution was also proved. On the other hand, both weak and strong
maximum principles were recently obtained for time-fractional diffusion
equations with Riemann-Liouville derivatives (see Al-Refai and Luchko
\cite{AL14}).

In this paper, we are interested in improving the maximum principle for
fractional diffusion equations with Caputo derivatives. Based on the weak
maximum principle obtained in \cite{L09} (see Lemma \ref{lem-wmp}), first we
establish a strong maximum principle for the initial-boundary value problem
\eqref{eq-ibvp-u}, which is slightly weaker than that for the parabolic case.

\begin{thm}\label{thm-smp}
Let $a\in L^2(\Om)$ satisfy $a\ge0$ and $a\not\equiv0$, $F=0$, and $u$ be the
solution to \eqref{eq-ibvp-u} with $d\le3$. Then for any $x\in\Om$, the set
$\cE_x:=\{t>0;\,u(x,t)\le0\}$ is at most a finite set.
\end{thm}

\begin{rem}
(a) By the Sobolev embedding and Lemma \ref{lem-sy11}(a) in Section
\ref{sec-pre}, we see $u\in C(\ov\Om\times(0,\infty))$ in Theorem \ref{thm-smp}
and thus the set $\cE_x$ is well-defined. According to the weak maximum
principle, it reveals that $\cE_x$ is actually the set of zero points of
$u(x,t)$ as a function of $t$, and Theorem \ref{thm-smp} asserts the strict
positivity of $u(x,t)$ for all $x\in\Om$ and almost all $t>0$ except for the
finite set $\cE_x$.

(b) Note that we have stated Theorem \ref{thm-smp} for spatial dimensions
$d\le3$. This can be generalized to arbitrary $d$ provided that our initial
data $a$ has sufficient regularity to allow a pointwise definition. For $d>3$
this will mean restricting $a$ in a subset of $L^2(\Om)$. If this is done, then
the set $\cE_x$ in Theorem \ref{thm-smp} will again be well-defined and, in
addition, in \eqref{eq-arg} we can observe that $\cA^3$ can be replaced by any
higher power $k$ necessary since the crucial requirement of
$C_0^\infty(\om)\subset\cD(\cA^k)$ is satisfied. However, to keep the
exposition simpler, we shall make the restriction $d\le3$ throughout the
remainder of the paper.

(c) It is an immediate consequence of Theorem \ref{thm-smp} that
$u>0$ a.e.\! in $\Om\times(0,\infty)$. To see this, we investigate the set
$D:=\{(x,t)\in\Om\times(0,\infty);\,u(x,t)\le0\}$ and notice
$D\cap(\{x\}\times(0,\infty))=\cE_x$. Since the characteristic function
$\chi_{\cE_x}=0$ a.e.\! in $(0,\infty)$ by Theorem \ref{thm-smp}, it follows
from Fubini's theorem that
\[
|D|=\int_{\Om\times(0,\infty)}\chi_D(x,t)\,\rd x\rd t
=\int_\Om\int_0^\infty\chi_{\cE_x}(t)\,\rd t\rd x=0,
\]
where $|\cdot|$ denotes the Lebesgue measure.

(d) If the inhomogeneous term $F$ in Theorem \ref{thm-smp} is allowed to be
non-negative in $\Om\times(0,T)$, then it follows immediately from the weak
maximum principle that $u>0$ a.e.\! in $\Om\times(0,T)$.
\end{rem}

So far, we do not know if $\cE_x=\emptyset$ ($\forall\,x\in\Om$) although it
can be conjectured. Nevertheless, we can prove the following result.

\begin{coro}\label{coro-smp}
Let $a\in L^2(\Om)$ satisfy $a>0$ a.e.\! in $\Om$, $F=0$, and $u$ be the
solution to \eqref{eq-ibvp-u}. Then $u>0$ in $\Om\times(0,\infty)$.
\end{coro}

Theorem \ref{thm-smp} is a weaker result than our expected strong maximum
principle, but is sufficient for some application. Next we study an inverse
source problem for \eqref{eq-ibvp-u} under the assumption that the
inhomogeneous term $F$ takes the form of separation of variables.

\begin{prob}\label{prob-ISP}
Let $x_0\in\Om$ and $T>0$ be arbitrarily given, and $u$ be the solution to
\eqref{eq-ibvp-u} with $a=0$ and $F(x,t)=\rho(t)\,g(x)$. Provided that $g$ is
known, determine $\rho(t)\ (0\le t\le T)$ by the single point observation data
$u(x_0,t)\ (0\le t\le T)$.
\end{prob}

The above problem is concerned with the determination of the temporal component
$\rho$ in the inhomogeneous term $F(x,t)=\rho(t)\,g(x)$ in \eqref{eq-ibvp-u}.
The spatial component $g$ simulates e.g.\! a source of contaminants which may
be dangerous. Although $g$ is usually limited to a small region given by
$\supp\,g(\subset\subset\Om)$, its influence may expand wider because $\rho(t)$
is large. We are requested to determine the time-dependent magnitude by the
pointwise data $u(x_0,t)\ (0\le t\le T)$, where $x_0$ is understood as a
monitoring point.

For the case of $x_0\in\supp\,g$, we know the both-sided stability estimate as
well as the uniqueness for Problem \ref{prob-ISP} (see Sakamoto and Yamamoto
\cite{SY11}). For the case of $x_0\notin\supp\,g$, there were no published
results even on the uniqueness. From the practical viewpoints mentioned above,
it is very desirable that $x_0$ should be spatially far from the location of
the source, that is, the case $x_0\notin\supp\,g$ should be discussed for the
inverse problem.

As a direct application of Theorem \ref{thm-smp}, we can give an affirmative
answer for the uniqueness regarding Problem \ref{prob-ISP}.

\begin{thm}\label{thm-ISP}
Under the same settings in Problem $\ref{prob-ISP}$, we further assume that
$\rho\in C^1[0,T]$, $g\in\cD(\cA^\ve)$ with some $\ve>0\ ($see Section
$\ref{sec-pre}$ for the definition of $\cD(\cA^\ve))$, $g\ge0$ and
$g\not\equiv0$. Then $u(x_0,t)=0\ (0\le t\le T)$ implies
$\rho(t)=0\ (0\le t\le T)$.
\end{thm}

In the case of $x_0\notin\supp\,g$, the condition $g\ge0$ and $g\not\equiv0$ in
$\Om$ is essential for the uniqueness. In fact, as a simple counterexample, we
consider $\Om=(0,1)$ and $g(x)=\sin2\pi x$ in Problem \ref{prob-ISP}, where the
condition $g\ge0$ is not satisfied. In this case, it follows from \cite{SY11}
that
\[
u(x,t)=\int_0^ts^{\al-1}E_{\al,\al}(-4\pi^2s^\al)\,\rho(t-s)\,\rd s\sin2\pi x,
\]
where $E_{\al,\al}(\,\cdot\,)$ denotes the Mittag-Leffler function (see
\eqref{eq-def-ML}). It is readily seen that $u(1/2,t)=0$ for $t>0$ and any
$\rho\in C^1[0,T]$. In other words, the data at $x_0=1/2$ does not imply the
uniqueness for Problem \ref{prob-ISP}.

For the same kind of inverse source problems for parabolic equations, we refer
to Cannon and Esteva \cite{CE86}, Saitoh, Tuan and Yamamoto \cite{STY02,STY03}.

The rest of this paper is organized as follows. Section \ref{sec-pre}
introduces the notations and collects the existing results concerning problem
\eqref{eq-ibvp-u}. Sections \ref{sec-proof-smp} is devoted to the proofs of
Theorem \ref{thm-smp} and Corollary \ref{coro-smp}, and the proof of Theorem
\ref{thm-ISP} is given in Section \ref{sec-proof-ISP}.

\Section{Preliminaries}\label{sec-pre}

To start with, we fix some general settings and notations. Let $L^2(\Om)$ be a
usual $L^2$-space with the inner product $(\,\cdot\,,\,\cdot\,)$ and
$H^1_0(\Om)$, $H^2(\Om)$ denote the Sobolev spaces (see, e.g., Adams
\cite{A75}). Let $\{(\la_n,\vp_n)\}_{n=1}^\infty$ be the eigensystem of the
symmetric uniformly elliptic operator $\cA$ in \eqref{eq-ibvp-u} such that
$0<\la_1<\la_2\le\cdots$ (the multiplicity is also counted), $\la_n\to\infty$
as $n\to\infty$ and $\{\vp_n\}\subset H^2(\Om)\cap H^1_0(\Om)$ forms an
orthonormal basis of $L^2(\Om)$. Then we can define the fractional power
$\cA^\ga$ for $\ga\ge0$ as
\[
\cD(\cA^\ga)=\left\{f\in L^2(\Om);\,\sum_{n=1}^\infty|\la_n^\ga\,(f,\vp_n)|^2
<\infty\right\},\quad\cA^\ga f:=\sum_{n=1}^\infty\la_n^\ga\,(f,\vp_n)\,\vp_n,
\]
and $\cD(\cA^\ga)$ is a Hilbert space with the norm
\[
\|f\|_{\cD(\cA^\ga)}
=\left(\sum_{n=1}^\infty|\la_n^\ga\,(f,\vp_n)|^2\right)^{1/2}.
\]

For $1\le p\le\infty$ and a Banach space $X$, we say that $f\in L^p(0,T;X)$
provided
\[
\|f\|_{L^p(0,T;X)}:=\left\{\!\begin{alignedat}{2}
& \left(\int_0^T\|f(\,\cdot\,,t)\|_X^p\,\rd t\right)^{1/p}
& \quad & \mbox{if }1\le p<\infty\\
& \mathop{\mathrm{ess}\sup}_{0<t<T}\|f(\,\cdot\,,t)\|_X
& \quad & \mbox{if }p=\infty
\end{alignedat}\right\}<\infty.
\]
Similarly, for $0\le t_0<T$, we set
\[
\|f\|_{C([t_0,T];X)}:=\max_{t_0\le t\le T}\|f(\,\cdot\,,t)\|_X.
\]
In addition, we define
\[
C((0,T];X):=\bigcap_{0<t_0<T}C([t_0,T];X),\quad
C([0,\infty);X):=\bigcap_{T>0}C([0,T];X).
\]

To represent the explicit solution of \eqref{eq-ibvp-u}, we first recall the
Mittag-Leffler function (see, e.g., Podlubny \cite{P99} and Gorenflo et al.
\cite{GKMS14})
\begin{equation}\label{eq-def-ML}
E_{\al,\be}(z):=\sum_{k=0}^\infty\f{z^k}{\Ga(\al k+\be)}
\quad(z\in\BC,\ \al>0,\ \be\in\BR),
\end{equation}
which possesses the following properties.

\begin{lem}\label{lem-ML}
{\rm(a)} Let $0<\al<2$ and $\eta>0$. Then $E_{\al,1}(-\eta)>0$ and the
following expansion holds:
\[
E_{\al,1}(-\eta)=\f1{\Ga(1-\al)\,\eta}+O\left(\f1{\eta^2}\right)
\quad\mbox{as }\eta\to\infty.
\]

{\rm(b)} Let $0<\al<2$ and $\be\in\BR$ be arbitrary. Then there exists a
constant $C=C(\al,\be)>0$ such that
\[
|E_{\al,\be}(-\eta)|\le\f C{1+\eta}\quad(\eta\ge0).
\]

{\rm(c)} For any $\ell=0,1,2,\ldots$, there holds
\[
E_{\al,1+\ell\al}(z)=\f1{\Ga(1+\ell\al)}+z\,E_{\al,1+(\ell+1)\al}(z)
\quad(\al>0,\ z\in\BC).
\]

{\rm(d)} For $\la>0$ and $\al>0$, we have
\[
\f\rd{\rd t}E_{\al,1}(-\la\,t^\al)=-\la\,t^{\al-1}E_{\al,\al}(-\la\,t^\al)
\quad(t>0).
\]
\end{lem}

We mention that Lemma \ref{lem-ML}(a)--(b) are well-known results from
\cite[\S 1.2]{P99}, and (c)--(d) follow immediately from direct calculations by
definition \eqref{eq-def-ML}.

Regarding some important existing results of the solution to \eqref{eq-ibvp-u},
we state the following two lemmata for later use.

\begin{lem}\label{lem-sy11}
Fix $T>0$ arbitrarily. Concerning the solution $u$ to \eqref{eq-ibvp-u}, we
have$:$

{\rm(a)} Let $a\in L^2(\Om)$ and $F=0$. Then there exists a unique solution
$u\in C([0,T];L^2(\Om))\cap C((0,T];H^2(\Om)\cap H_0^1(\Om))$, which can be
represented as
\begin{equation}\label{eq-rep-v}
u(\,\cdot\,,t)=\sum_{n=1}^\infty E_{\al,1}(-\la_nt^\al)\,(a,\vp_n)\,\vp_n
\end{equation}
in $C([0,T];L^2(\Om))\cap C((0,T];H^2(\Om)\cap H_0^1(\Om))$, where
$\{(\la_n,\vp_n)\}_{n=1}^\infty$ is the eigensystem of $\cA$. Moreover, there
exists a constant $C=C(\Om,T,\al,\cA)>0$ such that
\begin{equation}\label{eq-est-H2}
\|u(\,\cdot\,,t)\|_{L^2(\Om)}\le C\|a\|_{L^2(\Om)},\quad
\|u(\,\cdot\,,t)\|_{H^2(\Om)}\le C\|a\|_{L^2(\Om)}\,t^{-\al}\quad(0<t\le T).
\end{equation}
In addition, $u:(0,T]\to H^2(\Om)\cap H^1_0(\Om)$ can be analytically extended
to a sector $\{z\in\BC;\,z\ne0,\ |\arg z|<\pi/2\}$.

{\rm(b)} Let $a=0$ and $F\in L^\infty(0,T;L^2(\Om))$. Then there exists a
unique solution $u\in L^2(0,T;$ $H^2(\Om)\cap H_0^1(\Om))$ such that
$\lim_{t\to0}\|u(\,\cdot\,,t)\|_{L^2(\Om)}=0$.
\end{lem}

We note that Lemma \ref{lem-sy11} almost coincides with
\cite[Theorem 2.1]{SY11}, but the regularity of the analyticity result stated
in Lemma \ref{lem-sy11}(a) is stronger. Indeed, one can improve the regularity
up to $H^2(\Om)$ by the same reasoning, but here we omit the details.

\begin{lem}[Weak maximum principle]\label{lem-wmp}
Let $a\in L^2(\Om)$ and $F\in L^\infty(0,T;L^2(\Om))$ be nonnegative, and $u$
be the solution to \eqref{eq-ibvp-u}. Then there holds $u\ge0$ a.e.\! in
$\Om\times(0,T)$.
\end{lem}

Since we impose a homogeneous Dirichlet boundary condition, Lemma \ref{lem-wmp}
is a special case of \cite[Theorem 3]{L09}, but our choices of the initial
data, the inhomogeneous term and the elliptic operator are more general.
Fortunately, the same argument still works in our settings, which indicates
Lemma \ref{lem-wmp} immediately. Again we omit the details here.

\Section{Proof of Theorem \ref{thm-smp} and Corollary
\ref{coro-smp}}\label{sec-proof-smp}

Now we proceed to the proof of the strong maximum principle. Throughout this
section, we concentrate on the homogeneous problem, that is,
\begin{equation}\label{eq-ibvp-v}
v_a\begin{cases}
\pa_t^\al v+\cA v=0 & \mbox{in }\Om\times(0,\infty),\\
v=a & \mbox{in }\Om\times\{0\},\\
v=0 & \mbox{on }\pa\Om\times(0,\infty),
\end{cases}
\end{equation}
where we emphasize the dependency of the solution upon the initial data $a$ by
denoting the solution as $v_a$.

To begin with, we investigate the Green function of problem \eqref{eq-ibvp-v}.
Using the Mittag-Leffler function and the eigensystem
$\{(\la_n,\vp_n)\}_{n=1}^\infty$, for $N\in\BN$ we set
\[
G_N(x,y,t):=\sum_{n=1}^NE_{\al,1}(-\la_nt^\al)\,\vp_n(x)\,\vp_n(y)
\quad(x,y\in\Om,\ t>0).
\]
According to Lemma \ref{lem-sy11}(a), there holds
\[
v_a(x,t)=\lim_{N\to\infty}\int_\Om G_N(x,y,t)\,a(y)\,\rd y
\]
in $C([0,\infty);L^2(\Om))\cap C((0,\infty);H^2(\Om)\cap H_0^1(\Om))$ for any
$a\in L^2(\Om)$ and any $t>0$. Therefore, for any fixed $x\in\Om$ and $t>0$, we
see that $v_a(x,t)$ is defined pointwisely and thus $G_N(x,\,\cdot\,,t)$ is
weakly convergent to
\begin{equation}\label{eq-def-G}
G(x,y,t):=\sum_{n=1}^\infty E_{\al,1}(-\la_nt^\al)\,\vp_n(x)\,\vp_n(y)
\end{equation}
as a series with respect to $y$. In particular, we obtain
$G(x,\,\cdot\,,t)\in L^2(\Om)$ for all $x\in\Om$ and all $t>0$. Moreover, the
solution to \eqref{eq-ibvp-v} can be represented as
\begin{equation}\label{eq-rep-G}
v_a(x,t)=\int_\Om G(x,y,t)\,a(y)\,\rd y\quad(x\in\Om,\ t>0).
\end{equation}

Next we show that for arbitrarily fixed $x\in\Om$ and $t>0$,
$G(x,\,\cdot\,,t)\ge0$ a.e.\! in $\Om$. Actually, assume that on the contrary
there exist $x_1\in\Om$ and $t_1>0$ such that the Lebesgue measure of the
subdomain $\om:=\{G(x_1,\,\cdot\,,t_1)<0\}\subset\Om$ is positive. Then it is
readily seen that
\begin{equation}\label{eq-neg}
v_{\chi_\om}(x_1,t_1)=\int_\Om G(x_1,y,t_1)\,\chi_\om(y)\,\rd y<0,
\end{equation}
where $\chi_\om$ is the characteristic function of $\om$ satisfying
$\chi_\om\in L^2(\Om)$ and $\chi_\om\ge0$. On the other hand, Lemma
\ref{lem-wmp} and \eqref{eq-rep-G} imply that $v_{\chi_\om}(x_1,t_1)\ge0$,
which contradicts with \eqref{eq-neg}. In summary, we have proved the
following lemma.

\begin{lem}\label{lem-Green}
Let $G(x,y,t)$ be the Green function defined in \eqref{eq-def-G}. Then for
arbitrarily fixed $x\in\Om$ and $t>0$, we have
\[
G(x,\,\cdot\,,t)\in L^2(\Om)\quad\mbox{and}\quad
G(x,\,\cdot\,,t)\ge0\ \mbox{a.e.\! in }\Om.
\]
\end{lem}

Now we are well prepared to prove the strong maximum principle.

\begin{proof}[Proof of Theorem $\ref{thm-smp}$]
We deal with the homogeneous problem \eqref{eq-ibvp-v} with the initial data
$a\in L^2(\Om)$ such that $a\ge0$ and $a\not\equiv0$. By Lemma
\ref{lem-sy11}(a) and the Sobolev embedding $H^2(\Om)\subset C(\ov\Om)$ for
$d\le3$, we have $v_a\in C(\ov\Om\times(0,\infty))$. Meanwhile, according to
the weak maximum principle stated in Lemma \ref{lem-wmp}, there holds
$v_a\ge0$ in $\Om\times(0,\infty)$, indicating
\[
\cE_x:=\{t>0;\,v_a(x,t)\le0\}=\{t>0;\,v_a(x,t)=0\},
\]
that is, $\cE_x$ coincides with the zero point set of $v_a(x,t)$ as a function
of $t>0$.

Assume contrarily that there exists $x_0\in\Om$ such that the set $E_{x_0}$ is
not a finite set. Then $E_{x_0}$ contains at least an accumulation point
$t_*\in[0,\infty]$. We treat the cases of $t_*=\infty$, $t_*\in(0,\infty)$ and
$t_*=0$ separately.\medskip

{\it Case} 1. If $t_*=\infty$, then by definition there exists
$\{t_i\}_{i=1}^\infty\subset E_{x_0}$ such that $t_i\to\infty$ ($i\to\infty$)
and $u(x_0,t_i)=0$. Recall the explicit representation
\[
v_a(x_0,t)=\sum_{n=1}^\infty E_{\al,1}(-\la_nt^\al)\,(a,\vp_n)\,\vp_n(x_0).
\]
Then the asymptotic behavior described in Lemma \ref{lem-ML}(a) implies
\[
v_a(x_0,t)=\f1{\Ga(1-\al)\,t^\al}\sum_{n=1}^\infty
\f{(a,\vp_n)}{\la_n}\,\vp_n(x_0)+O\left(\f1{t^{2\al}}\right)\sum_{n=1}^\infty
\f{(a,\vp_n)}{\la_n^2}\,\vp_n(x_0)\quad\mbox{as }t\to\infty.
\]
Substituting $t=t_i$ with sufficiently large $i$ into the above expansion,
multiplying both sides by $t_i^\al$ and passing $i\to\infty$, we obtain
\[
b(x_0)=0,\quad\mbox{where}\quad
b:=\sum_{n=1}^\infty\f{(a,\vp_n)}{\la_n}\,\vp_n.
\]
Simple calculations reveal that $b$ satisfies the boundary value problem
\begin{equation}\label{eq-bvp}
\begin{cases}
\cA b=a\ge0 & \mbox{in }\Om,\\
b=0 & \mbox{on }\pa\Om.
\end{cases}
\end{equation}
Since the coefficient $c$ in the elliptic operator $\cA$ is non-negative, the
weak maximum principle for \eqref{eq-bvp} (see Gilbarg and Trudinger
\cite[Chapter 3]{GT01}) indicates $b\ge0$ in $\ov\Om\,$. Moreover, as $b$
attains its minimum at $x_0\in\Om$, the strong maximum principle for
\eqref{eq-bvp} implies $b\equiv\mathrm{const.}=0$ and thus $a=\cA b=0$, which
contradicts with the assumption $a\not\equiv0$. Therefore, $\infty$ cannot be
an accumulation point of $E_{x_0}$.\medskip

{\it Case} 2. Now suppose that the set of zeros $E_{x_0}$ admits an
accumulation point $t_*\in(0,\infty)$. By the analyticity of
$v_a:(0,\infty)\to H^2(\Om)\cap H^1_0(\Om)\subset C(\ov\Om)$, we see that
$v_a(x_0,t)$ is analytic with respect to $t>0$. Therefore, $v_a(x_0,t)$ should
vanish identically if its zero points accumulate at some finite and non-zero
point $t_*$. Then this case reduces to Case 1 and eventually result in a
contradiction.\medskip

{\it Case} 3. Since $v_a(x_0,t)$ is not analytic at $t=0$, we shall treat the
case of $t_*=0$ separately. Henceforth $C>0$ denotes generic constants
independent of $n\in\BN$ and $t\ge0$, which may change line by line.

By definition, there exists $\{t_i\}_{i=1}^\infty\subset E_{x_0}$ such that
$t_i\to0$ ($i\to\infty$) and, in view of the representation \eqref{eq-rep-G},
\[
v_a(x_0,t_i)=\int_\Om G(x_0,y,t_i)\,a(y)\,\rd y=0\quad(i=1,2,\ldots).
\]
Since $G(x_0,\,\cdot\,,t_i)\ge0$ by Lemma \ref{lem-Green} and $a\ge0$, we
deduce $G(x_0,y,t_i)\,a(y)=0$ for all $i=1,2,\ldots$ and almost all $y\in\Om$.
Since $a\not\equiv0$, it follows that $G(x_0,\,\cdot\,,t_i)$ should vanish in
the subdomain $\om:=\{a>0\}$ whose Lebesgue measure is positive. By the
representation \eqref{eq-def-G}, this indicates
\begin{equation}\label{eq-van-1}
\sum_{n=1}^\infty E_{\al,1}(-\la_nt_i^\al)\,\vp_n(x_0)\,\vp_n=0
\quad\mbox{a.e.\! in }\om\ (i=1,2,\ldots).
\end{equation}

Now we choose $\psi\in C^\infty_0(\om)$ arbitrarily as the initial data of
\eqref{eq-ibvp-v} and investigate
\begin{equation}\label{eq-van-def}
v_\psi(x_0,t)=\int_\Om G(x_0,y,t)\,\psi(y)\,\rd y=\sum_{n=1}^\infty
E_{\al,1}(-\la_nt^\al)\,(\psi,\vp_n)\,\vp_n(x_0).
\end{equation}
For later convenience, we abbreviate $\psi_n:=(\psi,\vp_n)\,\vp_n(x_0)$. We
shall show that the series in \eqref{eq-van-def} is convergent in
$C[0,\infty)$. In fact, by the Sobolev embedding $H^2(\Om)\subset C(\ov\Om)$
for $d\le3$, first we estimate
\[
|\vp_n(x_0)|\le C\|\vp_n\|_{H^2(\Om)}\le C\|\cA\vp_n\|_{L^2(\Om)}\le C\la_n.
\]
Next, since $\psi\in C^\infty_0(\om)\subset\cD(\cA^3)$, we have
\begin{equation}\label{eq-arg}
|(\psi,\vp_n)|=\f{|(\cA^3\psi,\vp_n)|}{\la_n^3}
\le\f{\|\cA^3\psi\|_{L^2(\om)}\|\vp_n\|_{L^2(\Om)}}{\la_n^3}
\le\f{C\|\psi\|_{C^6(\ov\om)}}{\la_n^3}.
\end{equation}
On the other hand, we know $\la_n\sim n^{2/d}$ as $n\to\infty$ (see, e.g.,
Courant and Hilbert \cite{CH53}). Therefore, the combination of the above
estimates yields
\[
|E_{\al,1}(-\la_nt^\al)\,\psi_n|\le C\|\psi\|_{C^6(\ov\om)}\la_n^{-2}
\le C\,n^{-4/d}\quad\mbox{as }n\to\infty,
\]
where the boundedness of $E_{\al,1}(-\la_nt^\al)$ is guaranteed by Lemma
\ref{lem-ML}(b). Since the restriction $d\le3$ gives $4/d>1$, we obtain
\[
\sum_{n=1}^\infty|E_{\al,1}(-\la_nt^\al)\,\psi_n|<\infty
\quad(\forall\,t\ge0,\ \forall\,\psi\in C_0^\infty(\om)),
\]
which indicates that $v_\psi(x_0,t)$ is well-defined in the sense of
$C[0,\infty)$. Meanwhile, by the same reasoning as that for \eqref{eq-arg}, for
any $\ell=0,1,2,\ldots$ we estimate
\[
|(\psi,\vp_n)|=\f{|(\cA^{\ell+3}\psi,\vp_n)|}{\la_n^{\ell+3}}
\le\f{C\|\psi\|_{C^{2(\ell+3)}(\ov\om)}}{\la_n^{\ell+3}},
\]
implying
\begin{equation}\label{eq-cov1}
\sum_{n=1}^\infty|\la_n^\ell\,\psi_n|
\le C\|\psi\|_{C^{2(\ell+3)}(\ov\om)}\sum_{n=1}^\infty\la_n^{-2}<\infty
\quad(\forall\,\ell=0,1,\ldots,\ \forall\,\psi\in C_0^\infty(\om)).
\end{equation}
Moreover, since $E_{\al,\be}(-\eta)$ is uniformly bounded for all $\eta\ge0$
and all $\be>0$ by Lemma \ref{lem-ML}(b), we further have
\begin{equation}\label{eq-cov2}
\sum_{n=1}^\infty|\la_n^\ell\,E_{\al,\be}(-\la_nt^\al)\,\psi_n|<\infty
\quad(\forall\,\ell=0,1,\ldots,\ \forall\,\be>0,\ \forall\,t\ge0,
\ \forall\,\psi\in C_0^\infty(\om)).
\end{equation}

Utilizing Lemma \ref{lem-ML}(c) with $\ell=0$, we treat $v_\psi(x_0,t)$ as
\[
v_\psi(x_0,t)=\sum_{n=1}^\infty\psi_n
-t^\al\sum_{n=1}^\infty\la_n\,E_{\al,1+\al}(-\la_nt^\al)\,\psi_n,
\]
where the boundedness of the involved summations were verified in
\eqref{eq-cov1}--\eqref{eq-cov2}. Taking $t=t_i$ and passing $i\to\infty$, we
obtain $\sum_{n=1}^\infty\psi_n=0$, implying
\[
v_\psi(x_0,t)
=\sum_{n=1}^\infty(-\la_nt^\al)\,E_{\al,1+\al}(-\la_nt^\al)\,\psi_n.
\]
For $t>0$, we divide the above equality by $-t^\al$ and take $\ell=1$ in Lemma
\ref{lem-ML}(c) to deduce
\[
\f{v_\psi(x_0,t)}{-t^\al}=\f1{\Ga(1+\al)}\sum_{n=1}^\infty\la_n\psi_n
-t^\al\sum_{n=1}^\infty\la_n^2\,E_{\al,1+2\al}(-\la_nt^\al)\,\psi_n.
\]
Again, we take $t=t_i$ and pass $i\to\infty$ to get
$\sum_{n=1}^\infty\la_n\psi_n=0$ and thus
\[
v_\psi(x_0,t)
=\sum_{n=1}^\infty(-\la_nt^\al)^2E_{\al,1+2\al}(-\la_nt^\al)\,\psi_n.
\]
Repeating the same process, we can show by induction that
\begin{equation}\label{eq-induct}
v_\psi(x_0,t)=\sum_{n=1}^\infty(-\la_nt^\al)^\ell
E_{\al,1+\ell\al}(-\la_nt^\al)\,\psi_n\quad(\forall\,\ell=0,1,2,\ldots).
\end{equation}
Actually, suppose that \eqref{eq-induct} holds for some $\ell\in\BN$. For
$t>0$, we divide \eqref{eq-induct} by $(-t^\al)^\ell$ and apply Lemma
\ref{lem-ML}(c) to deduce
\[
\f{v_\psi(x_0,t)}{(-t^\al)^\ell}=\f1{\Ga(1+\ell\al)}\sum_{n=1}^\infty
\la_n^\ell\psi_n-t^\al\sum_{n=1}^\infty
\la_n^{\ell+1}E_{\al,1+(\ell+1)\al}(-\la_nt^\al)\,\psi_n,
\]
where the boundedness of the involved summations follows from
\eqref{eq-cov1}--\eqref{eq-cov2}. Taking $t=t_i$ and passing $i\to\infty$, we
obtain $\sum_{n=1}^\infty\la_n^\ell\psi_n=0$ and thus
\[
v_\psi(x_0,t)=\sum_{n=1}^\infty
(-\la_nt^\al)^{\ell+1}E_{\al,1+(\ell+1)\al}(-\la_nt^\al)\,\psi_n.
\]

Now it suffices to prove
\begin{equation}\label{eq-lim}
\lim_{\ell\to\infty}\sum_{n=1}^\infty(-\la_nt^\al)^\ell
E_{\al,1+\ell\al}(-\la_nt^\al)\,\psi_n=0\quad(\forall\,t\ge0).
\end{equation}
In fact, writing $\eta=\la_nt^\al$, it turns out that
\[
(-\la_nt^\al)^\ell E_{\al,1+\ell\al}(-\la_nt^\al)
=(-\eta)^\ell\sum_{k=0}^\infty\f{(-\eta)^k}{\Ga(\al k+1+\ell\al)}
=\sum_{k=\ell}^\infty\f{(-\eta)^k}{\Ga(\al k+1)}
\]
coincides with the summation after the $\ell$th term in the series by which
$E_{\al,1}(-\eta)$ is defined. Noting that the series is uniformly convergent
with respect to $\eta\ge0$, we obtain
\[
\lim_{\ell\to\infty}(-\la_nt^\al)^\ell E_{\al,1+\ell\al}(-\la_nt^\al)=0
\quad(\forall\,n=1,2,\ldots,\ \forall\,t\ge0),
\]
which, together with the boundedness of $\sum_{n=1}^\infty|\psi_n|$, yields
\eqref{eq-lim} immediately. Since $v_\psi(x_0,t)$ is independent of $\ell$, by
\eqref{eq-induct} we eventually conclude
\begin{equation}\label{eq-van-2}
v_\psi(x_0,t)=\sum_{n=1}^\infty E_{\al,1}(-\la_nt^\al)\,\psi_n=0
\quad(t\ge0,\ \forall\,\psi\in C_0^\infty(\om)).
\end{equation}

Recall that the Laplace transform of $E_{\al,1}(-\la_nt^\al)$ reads (see, e.g.,
Podlubny \cite[p.21]{P99})
\[
\int^\infty_0\e^{-zt}E_{\al,1}(-\la_nt^\al)\,\rd t=\f{z^{\al-1}}{z^\al+\la_n}
\quad(\rRe\,z>\la_n^{1/\al}),
\]
which is analytically extended to $\rRe\,z>0$. Since the series in
\eqref{eq-van-2} converges in $C[0,\infty)$, we can take the Laplace transform
with respect to $t$ in \eqref{eq-van-2} to derive
\[
z^{\al-1}\sum_{n=1}^\infty\f{\psi_n}{z^\al+\la_n}=0
\quad(\rRe\,z>0,\ \forall\,\psi\in C_0^\infty(\om)),
\]
that is,
\begin{equation}\label{eq-van-3}
\sum_{n=1}^\infty\f{\psi_n}{\ze+\la_n}=0
\quad(\rRe\,\ze>0,\ \forall\,\psi\in C_0^\infty(\om)).
\end{equation}
By a similar argument for the convergence of \eqref{eq-van-def}, we see that
the above series is also convergent in any compact set in
$\BC\setminus\{-\la_n\}_{n=1}^\infty$, and the analytic continuation in $\ze$
yields that \eqref{eq-van-3} holds for
$\ze\in\BC\setminus\{-\la_n\}_{n=1}^\infty$. Especially, since the first
eigenvalue $\la_1$ is single, we can choose a small circle around $-\la_1$
which does not contain $-\la_n$ ($n\ge2$). Integrating \eqref{eq-van-3} on this
circle yields
\[
\psi_1=(\vp_1,\psi)\,\vp_1(x_0)=0\quad(\forall\,\psi\in C_0^\infty(\om)).
\]
Since $\psi\in C_0^\infty(\om)$ is arbitrarily chosen, there should hold
$\vp_1(x_0)\,\vp_1=0$ a.e.\! in $\om$. However, this contradicts with the
strict positivity of the first eigenfunction $\vp_1$ (see, e.g., Evans
\cite{E10}). Therefore, $t_*=0$ cannot be an accumulation point of $E_{x_0}$.

In summary, for any $x\in\Om$, we have excluded all the possibilities for
$\cE_x$ to possess any accumulation point, indicating that $\cE_x$ is at most a
finite set.
\end{proof}

Taking advantage of the Green function introduced in \eqref{eq-def-G}, it is
straightforward to demonstrate Corollary \ref{coro-smp}.

\begin{proof}[Proof of Corollary $\ref{coro-smp}$]
Recall that the solution $v_a$ allows a pointwise definition if
$a\in L^2(\Om)$, and $v_a\ge0$ in $\ov\Om\times(0,\infty)$ by Lemma
\ref{lem-wmp}. Assume contrarily that there exists $x_0\in\Om$ and $t_0>0$ such
that $v_g(x_0,t_0)=0$. Employing the representation \eqref{eq-rep-G}, we see
$\int_\Om G(x_0,y,t_0)\,a(y)\,\rd y=0$. Since $G(x_0,\,\cdot\,,t_0)\ge0$ by
Lemma \ref{lem-Green} and $a>0$, there should be $G(x_0,\,\cdot\,,t_0)=0$, that
is,
\[
\sum_{n=1}^\infty E_{\al,1}(\la_nt_0^\al)\,\vp_n(x_0)\,\vp_n=0
\quad\mbox{in }\Om.
\]
Since $\{\vp_n\}$ is a complete orthonormal basis in $L^2(\Om)$, we obtain
$E_{\al,1}(-\la_nt_0^\al)\,\vp_n(x_0)=0$ for all $n=1,2,\ldots$, especially,
$E_{\al,1}(-\la_1t_0^\al)\,\vp_1(x_0)=0$. However, it is impossible because
$E_{\al,1}(-\la_1t_0^\al)>0$ according to Lemma \ref{lem-ML}(a) and
$\vp_1(x_0)>0$. Therefore, such a pair $(x_0,t_0)$ cannot exist and we complete
the proof.
\end{proof}

\Section{Proof of Theorem \ref{thm-ISP}}\label{sec-proof-ISP}

Now we turn to the proof of Theorem \ref{thm-ISP}, i.e., the uniqueness for
Problem \ref{prob-ISP}. Recall that the initial-boundary value problem under
consideration is
\begin{equation}\label{eq-ibvp-w}
\begin{cases}
\pa_t^\al u(x,t)+\cA u(x,t)=\rho(t)\,g(x) & (x\in\Om,\ 0<t\le T),\\
u(x,0)=0 & (x\in\Om),\\
u(x,t)=0 & (x\in\pa\Om,\ 0<t\le T),
\end{cases}
\end{equation}
where $\rho\in C^1[0,T]$, $g\in\cD(\cA^\ve)$ with some $\ve>0$, $g\ge0$ and
$g\not\equiv0$. First we establish the following fractional Duhamel's principle
for the fractional diffusion equation.

\begin{lem}\label{lem-Duhamel}
Let $u$ be the solution to \eqref{eq-ibvp-w}, where $\rho\in C^1[0,T]$ and
$g\in\cD(\cA^\ve)$ with some $\ve>0$. Then $u$ allows the representation
\[
u(\,\cdot\,,t)=\int_0^t\mu(t-s)\,v_g(\,\cdot\,,s)\,\rd s\quad(0<t\le T),
\]
where $v_g$ solves the homogeneous problem \eqref{eq-ibvp-v} with $g$ as the
initial data, and
\begin{equation}\label{eq-def-mu}
\mu(t):=\f1{\Ga(\al)}\f\rd{\rd t}\int_0^t\f{\rho(s)}{(t-s)^{1-\al}}\,\rd s
\quad(0<t\le T).
\end{equation}
\end{lem}

\begin{proof}
Henceforth $C>0$ denotes generic constants independent of the choice of $g$.
First, since $\rho\,g\in L^\infty(0,T;L^2(\Om))$, Lemma \ref{lem-sy11}(b)
indicates that $u\in L^2(0,T;H^2(\Om)\cap H_0^1(\Om))$ and
$\lim_{t\to0}\|u(\,\cdot\,,t)\|_{L^2(\Om)}=0$. By setting
\begin{equation}\label{eq-def-wtu}
\wt u(\,\cdot\,,t):=\int_0^t\mu(t-s)\,v(\,\cdot\,,s)\,\rd s,
\end{equation}
we shall demonstrate
\[
u=\wt u\quad\mbox{in }L^2(0,T;H^2(\Om)\cap H_0^1(\Om)),
\quad\lim_{t\to0}\|\wt u(\,\cdot\,,t)\|_{L^2(\Om)}=0.
\]

Since $\rho\in C^1[0,T]$, simple calculations for \eqref{eq-def-mu} yield
\begin{align}
\mu(t) & =\f1{\Ga(\al)}\f\rd{\rd t}\left(-\left.\f{\rho(s)\,(t-s)^\al}\al
\right|_{s=0}^{s=t}+\f1\al\int_0^t(t-s)^\al\rho'(s)\,\rd s\right)\nonumber\\
& =\f1{\Ga(\al)}\f\rd{\rd t}\left(\f{\rho(0)}\al\,t^\al
+\f1\al\int_0^t(t-s)^\al\rho'(s)\,\rd s\right)\nonumber\\
& =\f1{\Ga(\al)}\left(\f{\rho(0)}{t^{1-\al}}
+\int_0^t\f{\rho'(s)}{(t-s)^{1-\al}}\,\rd s\right),\label{eq-rep-mu}
\end{align}
which further implies
\begin{equation}\label{eq-est-mu}
\mu\in L^1(0,T),\quad|\mu(t)|\le C\,t^{\al-1}\ (0<t\le T).
\end{equation}

Regarding the solution $v_g$ to \eqref{eq-ibvp-v}, since $g\in L^2(\Om)$, the
application of estimate \eqref{eq-est-H2} in Lemma \ref{lem-sy11}(a) yields
\[
\|v_g(\,\cdot\,,t)\|_{L^2(\Om)}\le C\|g\|_{L^2(\Om)},\quad
\|v_g(\,\cdot\,,t)\|_{H^2(\Om)}\le C\|g\|_{L^2(\Om)}\,t^{-\al}\quad(0<t\le T).
\]
By the relation \eqref{eq-def-wtu} and the estimate in \eqref{eq-est-mu}, for
$0<t\le T$ we estimate
\begin{align*}
\|\wt u(\,\cdot\,,t)\|_{L^2(\Om)}
& \le\int_0^t|\mu(t-s)|\|v_g(\,\cdot\,,s)\|_{L^2(\Om)}\,\rd s
\le C\|g\|_{L^2(\Om)}\int_0^ts^{\al-1}\,\rd s\le C\|g\|_{L^2(\Om)}\,t^\al,\\
\|\wt u(\,\cdot\,,t)\|_{H^2(\Om)}
& \le\int_0^t|\mu(t-s)|\|v_g(\,\cdot\,,s)\|_{H^2(\Om)}\,\rd s
\le C\|g\|_{L^2(\Om)}\int_0^t(t-s)^{\al-1}s^{-\al}\,\rd s\\
& \le C\|g\|_{L^2(\Om)},
\end{align*}
indicating
\[
\wt u\in L^\infty(0,T;H^2(\Om)\cap H_0^1(\Om))\subset
L^2(0,T;H^2(\Om)\cap H_0^1(\Om)),
\quad\lim_{t\to0}\|\wt u(\,\cdot\,,t)\|_{L^2(\Om)}=0.
\]
On the other hand, according to the explicit representation \eqref{eq-rep-v}
and Lemma \ref{lem-ML}(d), we obtain
\[
\pa_tv_g(\,\cdot\,,t)
=-t^{\al-1}\sum_{n=1}^\infty E_{\al,\al}(-\la_nt^\al)\,(g,\vp_n)\,\vp_n.
\]
By Lemma \ref{lem-ML}(b) and the fact $g\in\cD(\cA^\ve)$ with $\ve>0$, we
estimate for $0<t\le T$ that
\begin{align}
\|\pa_tv_g(\,\cdot\,,t)\|_{L^2(\Om)} & \le t^{2(\al-1)}\sum_{n=1}^\infty
|E_{\al,\al}(-\la_nt^\al)\,(g,\vp_n)|^2\nonumber\\
& =t^{2(\al-1)}\sum_{n=1}^\infty\left|\la_n^{1-\ve}
E_{\al,\al}(-\la_nt^\al)\right|^2|\la_n^\ve\,(g,\vp_n)|^2\nonumber\\
& \le\left(C\,t^{\al\ve-1}\right)^2\sum_{n=1}^\infty
\left|\f{(\la_nt^\al)^{1-\ve}}{1+\la_nt^\al}\right|^2|\la_n^\ve\,(g,\vp_n)|^2
\le\left(C\|g\|_{\cD(\cA^\ve)}\,t^{\al\ve-1}\right)^2.\label{eq-est-wtut}
\end{align}

In order to show $u=\wt u$, it suffices to verify that $\wt u$ also solves the
initial-boundary value problem \eqref{eq-ibvp-w} which possesses a unique
solution (see Lemma \ref{lem-sy11}(b)). To calculate $\pa_t^\al\wt u$, first we
formally calculate
\begin{equation}\label{eq-wtut}
\pa_t\wt u(\,\cdot\,,t)=\pa_t\int_0^t\mu(s)\,v_g(\,\cdot\,,t-s)\,\rd s
=\int_0^t\mu(s)\,\pa_tv_g(\,\cdot\,,t-s)\,\rd s+\mu(t)\,g.
\end{equation}
Then we employ \eqref{eq-est-mu} and \eqref{eq-est-wtut} to estimate
\begin{align*}
\|\pa_t\wt u(\,\cdot\,,t)\|_{L^2(\Om)} & \le\int_0^t|\mu(t-s)|
\|\pa_sv_g(\,\cdot\,,s)\|_{L^2(\Om)}\,\rd s+|\mu(t)|\|g\|_{L^2(\Om)}\\
& \le C\|g\|_{\cD(\cA^\ve)}\int_0^t(t-s)^{\al-1}s^{\al\ve-1}\,\rd s
+C\|g\|_{L^2(\Om)}\,t^{\al-1}\\
& \le C\|g\|_{\cD(\cA^\ve)}\,t^{\al(1+\ve)-1}+C\|g\|_{L^2(\Om)}\,t^{\al-1}
\le C\|g\|_{\cD(\cA^\ve)}\,t^{\al-1}\quad(0<t\le T),
\end{align*}
implying that the above differentiation makes sense in $L^2(\Om)$ for
$0<t\le T$. By definition, we have
\begin{align*}
& \pa_t^\al\wt u(\,\cdot\,,t)=\f1{\Ga(1-\al)}\int_0^t
\f{\pa_s\wt u(\,\cdot\,,s)}{(t-s)^\al}\,\rd s=I_1+I_2\,g,\quad\mbox{where}\\
& I_1:=\f1{\Ga(1-\al)}\int_0^t\f1{(t-s)^\al}\int_0^s
\mu(\tau)\,\pa_sv(\,\cdot\,,s-\tau)\,\rd\tau\rd s,\\
& I_2:=\f1{\Ga(1-\al)}\int_0^t\f{\mu(s)}{(t-s)^\al}\,\rd s.
\end{align*}
The governing equation \eqref{eq-ibvp-v} for $v_g$ and formula
\eqref{eq-rep-mu} for $\mu$ imply respectively
\begin{align}
I_1 & =\f1{\Ga(1-\al)}\int_0^t\mu(\tau)\int_\tau^t
\f{\pa_sv(\,\cdot\,,s-\tau)}{(t-s)^\al}\,\rd s\rd\tau\nonumber\\
& =\int_0^t\mu(\tau)\left(\f1{\Ga(1-\al)}\int_0^{t-\tau}
\f{\pa_sv(\,\cdot,,s)}{((t-\tau)-s)^\al}\,\rd s\right)\rd\tau
=\int_0^t\mu(\tau)\,\pa_t^\al v(\,\cdot\,,t-\tau)\,\rd\tau\nonumber\\
& =-\int_0^t\mu(\tau)\,\cA v(\,\cdot\,,t-\tau)\,\rd\tau
=-\cA\int_0^t\mu(\tau)\,v(\,\cdot\,,t-\tau)\,\rd\tau
=-\cA\wt u(\,\cdot\,,t),\nonumber\\
I_2 & =\f1{\Ga(1-\al)\,\Ga(\al)}\int_0^t\f1{(t-s)^\al}
\left(\f{\rho(0)}{s^{1-\al}}+\int_0^s
\f{\rho'(\tau)}{(s-\tau)^{1-\al}}\,\rd\tau\right)\rd s\nonumber\\
& =\f1{\Ga(1-\al)\,\Ga(\al)}\left(\rho(0)\int_0^t
\f{\rd s}{(t-s)^\al\,s^{1-\al}}+\int_0^t\rho'(\tau)\int_\tau^t
\f{\rd s}{(t-s)^\al(s-\tau)^{1-\al}}\,\rd\tau\right)\nonumber\\
& =\rho(0)+\int_0^t\rho'(\tau)\,\rd\tau=\rho(t).\label{eq-rev-mu}
\end{align}
Therefore, we conclude $\pa_t^\al\wt u+\cA\wt u=\rho\,g$ and the proof is
completed.
\end{proof}

At this stage, we can proceed to show Theorem \ref{thm-ISP} by applying the
established strong maximum principle.

\begin{proof}[Completion of Proof of Theorem $\ref{thm-ISP}$]
Let the conditions in the statement of Theorem \ref{thm-ISP} be valid, namely,
it is assumed that $\rho\in C^1[0,T]$, $g\in\cD(\cA^\ve)$ with $\ve>0$,
$g\ge0$, $g\not\equiv0$, and the solution $u$ to \eqref{eq-ibvp-w} vanishes in
$\{x_0\}\times[0,T]$ for some $x_0\in\Om$. According to the fractional
Duhamel's principle proved above, we have
\[
u(x_0,t)=\int_0^t\mu(t-s)\,v_g(x_0,s)\,\rd s=0\quad(0\le t\le T),
\]
where $\mu$ was defined in \eqref{eq-def-mu} and $v_g$ solves \eqref{eq-ibvp-v}
with the initial data $g$. Now the estimate \eqref{eq-est-H2} in Lemma
\ref{lem-sy11}(a) and the Sobolev embedding indicate
\[
|v_g(x_0,t)|\le C\|v_g(\,\cdot\,,t)\|_{H^2(\Om)}\le C\|g\|_{L^2(\Om)}\,t^{-\al}
\quad(0<t\le T)
\]
and thus $v_g(x_0,\,\cdot\,)\in L^1(0,T)$. Meanwhile, \eqref{eq-est-mu}
guarantees $\mu\in L^1(0,T)$. Therefore, the Titchmarsh convolution theorem
(see \cite{T26}) implies the existence of $T_1,T_2\ge0$ satisfying
$T_1+T_2\ge T$ such that $\mu(t)=0$ for almost all $t\in(0,T_1)$ and
$v_g(x_0,t)=0$ for all $t\in[0,T_2]$. However, since the initial data $g$ of
\eqref{eq-ibvp-v} satisfies $g\ge0$ and $g\not\equiv0$, Theorem \ref{thm-smp}
asserts that $v_g(x_0,\,\cdot\,)$ only admits at most a finite number of zero
points, i.e., $v_g(x_0,\,\cdot\,)>0$ a.e.\! in $(0,T)$. As a result, the only
possibility is $T_2=0$ and thus $T_1=T$, that is, $\mu=0$ a.e.\! in $(0,T)$.

Finally, it suffices to utilize the following reverse formula
\[
\rho(t)=\f1{\Ga(1-\al)}\int_0^t\f{\mu(s)}{(t-s)^\al}\,\rd s,
\]
which was obtained in \eqref{eq-rev-mu}. We apply Young's inequality to
conclude
\[
\|\rho\|_{L^1(0,T)}=\f1{\Ga(1-\al)}\left\|\int_0^t
\f{\mu(s)}{(t-s)^\al}\,\rd s\right\|_{L^1(0,T)}
\le\f{T^{1-\al}}{\Ga(2-\al)}\|\mu\|_{L^1(0,T)}=0,
\]
which finishes the proof.
\end{proof}

\end{document}